\documentclass[a4paper,reqno,oneside]{amsart}
\usepackage[foot]{amsaddr}

\usepackage{ifthen}
\usepackage[utf8]{inputenc}
\usepackage[english]{babel}
\usepackage{amsmath,amsfonts,amssymb,amsthm}
\usepackage{graphicx}
\usepackage[svgnames]{xcolor}
\usepackage{xcolor}
\usepackage{mathtools}
\usepackage{hyperref}
\usepackage{enumerate}

\usepackage{caption}
\usepackage{subcaption}
\usepackage{tikz}
\pgfdeclarelayer{edgelayer}
\pgfdeclarelayer{nodelayer}
\pgfsetlayers{edgelayer,nodelayer,main}
\newtheorem{theorem}{Theorem}
\newtheorem{lemma}[theorem]{Lemma}
\newtheorem{claim}[theorem]{Claim}
\newtheorem{proposition}[theorem]{Proposition}

\theoremstyle{definition}
\newtheorem{definition}[theorem]{Definition}

\theoremstyle{remark}

\newcommand{\ignore}[1]{} % for commenting large passages

\newcommand{\cB}{\ensuremath{\mathcal B}}

\newcommand{\cE}{\ensuremath{\mathcal E}}

\newcommand{\cH}{\ensuremath{\mathcal H}}

\newcommand{\cT}{\ensuremath{\mathcal T}}

\renewcommand{\phi}{\varphi}

\usepackage{latex_defs}
\usepackage{stmaryrd}

\begin{document}

\title{On the threshold for the Maker-Breaker $H$-game}

\author{Rajko Nenadov$^1$, Angelika Steger$^1$}
\thanks{$^1$Institute of Theoretical Computer Science, ETH Z\"urich, 8092 Z\"urich, Switzerland, Email: \{rnenadov$|$steger\}@inf.ethz.ch}

\author{Milo\v s Stojakovi\'c$^2$}
\thanks{$^2$Department of Mathematics and Informatics, University of Novi Sad, Serbia. Email:
milos.stojakovic@dmi.uns.ac.rs. Partly supported by Ministry of Education and Science, Republic of Serbia, and Provincial Secretariat for Science, Province of
Vojvodina.}

%\begin{center}

%\LARGE On the threshold for the Maker-Breaker $H$-game
%\vspace{7mm}

%\Large{
%\begin{center}
%\begin{tabular}{ccccc}
%Rajko Nenadov\textsuperscript{$\ast$} & \quad & Angelika Steger\textsuperscript{$\ast$} & \quad & Milo\v s Stojakovi\'c\textsuperscript{$\dagger$} 
%\end{tabular}
%\end{center}
%}
%\vspace{7mm}

%\large
%  \textsuperscript{$\ast$} Institute of Theoretical Computer Science \\
%  ETH Z\"urich, 8092 Z\"urich, Switzerland \\
%  {\small\{{\tt rnenadov|steger}\}{\tt @inf.ethz.ch}}
%\vspace{5mm}

%\large
%  \textsuperscript{$\dagger$} Department of Mathematics and Informatics  \\
%  University of Novi Sad, 21000 Novi Sad, Serbia \\
%  {\small{\tt milos.stojakovic@dmi.uns.ac.rs}}
%\vspace{5mm}

%\small

\begin{abstract}
We study the Maker-Breaker $H$-game played on the edge set of the random graph $G_{n,p}$. In this game two players, Maker and Breaker, alternately claim unclaimed edges of $G_{n,p}$, until all the edges are claimed. Maker wins if he claims all the edges of a copy of a fixed graph $H$; Breaker wins otherwise.  In this paper we show that, with the exception of trees and triangles, the threshold for an $H$-game is given by the threshold of the corresponding Ramsey property of $G_{n,p}$ with respect to the graph~$H$.
%Extending their result, we determine the threshold for a large class of graphs, namely those which contain a cycle and whose $2$-density is not determined by a $K_3$ subgraph. In particular, we prove that for every such graph $H$ there exist constants $c, C > 0$ such that
%$$
%\Pr[\Gnp \; \text{is Maker's win in} \; H\text{-game}] = \begin{cases}
%1 - o(1), & p \geq C n^{-1 / m_2(H)}\\
%o(1), & p \leq c n^{-1 / m_2(H)},
%\end{cases}
%$$
%where $m_2(H) = \max_{H' \subseteq H, v(H) \geq 3} \frac{e(H) - 1}{v(H) - 2}$ is the so-called $2$-density of $H$.

\end{abstract}

\maketitle

\section{Introduction}

Combinatorial games are games like Tic-Tac-Toe or Chess in which each player has perfect information and players move sequentially. Outcomes of such games can thus, at least in principle, be predicted by enumerating all possible ways in which the game may evolve. But, of course, such complete enumerations usually exceed available computing powers, which keeps these games interesting to study.

In this paper we take a look at a special class of combinatorial games, the so-called Maker-Breaker positional games. Given a finite set $X$ and a family $\cE$ of subsets of $X$, two players, Maker and Breaker, alternate in claiming unclaimed elements of $X$ until all the elements are claimed. Unless explicitly stated otherwise, Maker starts the game. Maker wins if he claims all elements of a set from $\cE$, and Breaker wins otherwise. The set $X$ is referred to as the \emph{board}, and the elements of $\cE$ as the \emph{winning sets}.

Given a (large) graph $G$ and a (small) graph $H$, the \emph{$H$-game} on $G$ is played on the board $E(G)$ and the winning sets are the edge sets of all copies of $H$ appearing in $G$ as subgraphs. So, Maker and Breaker alternately claim unclaimed edges of the graph $G$ until all the edges are claimed. Maker wins if he claims all the edges of a copy of $H$, otherwise Breaker wins.

Positional games played on edges of random graphs were first introduced and studied in~\cite{SS05}. Here we look at the $H$-game played on the random graph $G_{n,p}$, where $H$ is a fixed graph. More precisely, we aim at determining a threshold function $p_0=p_0(n,H)$ such that
$$
\lim_{n\rightarrow \infty} \Pr[\Gnp \; \text{is Maker's win in the} \; H\text{-game}] = \begin{cases}
1, & p \gg p_0(n,H),\\
0, & p \ll p_0(n,H).
\end{cases}
$$
For the case that $H$ is a clique such thresholds were recently obtained by M\"uller and Stojakovi\'c~\cite{MS}.
There is an easy intuitive argument for the location of such a threshold:
if the random graph $G_{n,p}$ is so sparse that w.h.p.\ it only contains few scattered copies of $H$ then this should be a Breaker's win. If on the other hand the graph contains many copies of $H$ that heavily overlap then this should make Maker's task easier. As it turns out, the same intuition can also be applied to the threshold for the Ramsey property of $\Gnp$, thus one should expect that the two are related. We formalize this as follows. %threshold of the $H$-game on $G_{n,p}$ is closely related to the threshold
%for the Ramsey property of $G_{n,p}$ that we define next. %In this paper we confirm this intuition.

For graphs $G$ and $H$ we denote by
$ G \rightarrow (H)^e_2 $
the property that every edge-coloring of $G$ with $2$ colors contains a copy of $H$ with all edges having the same color.
For a graph $G=(V,E)$ on at least three vertices, we let $d_2(G) := (|E| - 1)/(|V|-2)$ and denote by $m_2(G)$ the so-called \emph{2-density}, defined as
$m_2(G) = \max_{{J \subseteq G, v_J \geq 3}} d_2(J)$.
If $m_2(G) = d_2(G)$, we say that $G$ is \emph{2-balanced}, and if in addition $m_2(G) > d_2(J)$ for every subgraph $J \subset G$ with $v_J \geq 3$, we say that $G$ is \emph{strictly 2-balanced}.

The Ramsey property of random graphs $G_{n,p}$ is well understood, as the following theorem shows, cf.\ also~\cite{NS13} for a short proof.

\begin{theorem}[R\"odl, Ruci\'nski \cite{RR93,RR94,RR95}] \label{thm:rr}
Let  \($H$\) be a graph that is not a forest of stars or paths of length 3. Then there exist constants $c, C > 0$ such that
        \begin{equation*}
          \lim_{n\to \infty}\Pr[\Gnp \rightarrow (H)^e_2] =
          \begin{cases}           
            1,&\text{if \(p \geq Cn^{-1/m_2(H)}\)}, \\
            0,&\text{if \(p \leq cn^{-1/m_2(H)}\)}.
          \end{cases}
        \end{equation*}
\end{theorem}

Note that $p= n^{-1/m_2(F)}$ is the density where we expect that every edge is contained in roughly a constant number of copies of $H$. Thus, if $c$ is very small, the of copies of $H$ will be scattered. If on the other hand $C$ is big then these copies overlap so heavily that every coloring has to induce at least one monochromatic copy of $H$. 
%In other words, the threshold for the Ramsey property of $G_{n,p}$ satisfies exactly the desired properties for the threshold of the $H$-game on $G_{n,p}$ that we stated above.

In this paper, we show that this intuition indeed provides the correct answer for most graphs $H$.
%at least in the case expect of trees and triangles. More precisely, we show

\begin{theorem} \label{thm:main}
Let $H$ be a graph for which there exists $H' \subseteq H$ such that $d_2(H') = m_2(H)$, $H'$ is strictly $2$-balanced and it is not a tree or a triangle. Then there exist constants $c, C > 0$ such that
$$\lim_{n \rightarrow \infty} \Pr[\Gnp \; \text{is Maker's win in the} \; H\text{-game}] =
\begin{cases}
1, &p \geq Cn^{-1/m_2(H)}, \\
0, &p \leq cn^{-1/m_2(H)}.
\end{cases} $$
\end{theorem}

Next, we take a look at the graphs $H$ that are not covered by Theorem~\ref{thm:main}. For $H=K_3$ we have $m_2(K_3)=2$. Nevertheless, the threshold for the $K_3$-game is $n^{-5/9}$, cf.\ \cite{SS05}. The reason turns out to be that $K_5$ minus an edge is a Maker's win (which can be easily checked by hand) -- and this graph appears in $G_{n,p}$ w.h.p.\ whenever $p\gg n^{-5/9}$.

For graphs $H$ that contain a triangle, various things can happen. If their 2-density is above two, then they are covered by the above theorem. If $m_2(H)=2$ and $H$ contains a subgraph with 2-density exactly two that does not contain a triangle, then this case is also covered by the above theorem. Otherwise, the threshold can be placed almost arbitrarily between $n^{-5/9}$ and $n^{-1/2}$ while the 2-density of $H$ remains at 2, as our next theorem confirms. In particular, we show that there exists a class of graphs for which the threshold is not determined by the $2$-densest subgraph.

For a graph $H$, we denote by $H_P$ the graph obtained by adding  a path of length~$3$ between a vertex of  a $K_3$ and an arbitrary vertex of $H$, see Figure~\ref{fig:H_p}.

\begin{figure}[h!]
  \centering
  \tikzstyle{node}=[circle,fill=White,draw=Black,scale=0.5]
\tikzstyle{none}=[inner sep=0pt]
\tikzstyle{neutral}=[-]

\begin{tikzpicture}
	\begin{pgfonlayer}{nodelayer}
		\node [style=node] (0) at (-3.5, 2) {};
		\node [style=node] (1) at (-4.25, 1) {};
		\node [style=node] (2) at (-2.75, 1) {};
		\node [style=node] (3) at (-1.75, 1) {};
		\node [style=node] (4) at (-0.75, 1) {};
		\node [style=node] (5) at (0.25, 1) {};
		\node [style=none] (6) at (1.75, 2) {};
		\node [style=none] (7) at (1.25, 1.25) {H};
		\node [style=none] (8) at (1.75, 0.25) {};
	\end{pgfonlayer}
	\begin{pgfonlayer}{edgelayer}
		\draw [style=neutral] (0) to (2);
		\draw [style=neutral] (0) to (1);
		\draw [style=neutral] (1) to (2);
		\draw [style=neutral] (2) to (3);
		\draw [style=neutral] (3) to (4);
		\draw [style=neutral] (4) to (5);
		\draw [style=neutral, bend left=75, looseness=1.50] (5) to (6.center);
		\draw [style=neutral, bend left=90, looseness=1.75] (6.center) to (8.center);
		\draw [style=neutral, bend left=75, looseness=1.50] (8.center) to (5);
	\end{pgfonlayer}
\end{tikzpicture}
  \caption{Graph $H_P$} \label{fig:H_p}
\end{figure}
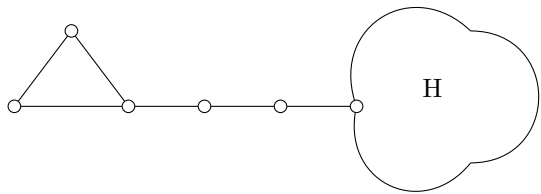

\begin{theorem}\label{thm:k3_extension}
Let $H$ be a graph which satisfies the conditions of Theorem~\ref{thm:main}. Then for $t = \min\{\frac 59,  1 / {m_2(H)}\}$ we have
$$
\lim_{n \rightarrow \infty} \Pr[\Gnp \; \text{is Maker's win in the} \; H_P\text{-game}] =
\begin{cases}
1, &p \gg n^{-t}, \\
0, &p \ll n^{-t}.
\end{cases}
$$
\end{theorem}

%In particular, if we have $9/5 < m_2(H) <2$, then $m_2(H_P)=2$, and from Theorem~\ref{thm:k3_extension} we get that the threshold for Maker's win in the $H_P$-game is $n^{-1 / {m_2(H)}}$.

Our paper is structured as follows. In the next section we collect some preliminaries. Then, in Sections~\ref{sec:main1}-\ref{sec:main3} we prove Theorem~\ref{thm:main}, while in Section~\ref{sec:add} we prove Theorem~\ref{thm:k3_extension}.

\section{Preliminaries}

In this section we collect some known properties about positional games, graph decompositions and random graphs. We follow the standard notation. In particular, for a graph $G$ and a subset $A \subseteq V(G)$, we denote with $N_G(A)$ the neighborhood of $A$ in $V(G) \setminus A$, i.e.
$$ N_G(A) := \left\{ v \in V(G) \setminus A \mid \exists a \in A \; \text{such that} \; \{v,a\} \in E(G) \right\}. $$
If the graph $G$ is clear from the context, we omit it in the subscript. Furthermore, for a graph $G$ we use $v_G$ and $e_G$ to denote the number of vertices and edges of $G$, respectively.

\subsection{Positional games}

For a Maker-Breaker game with the board $X$ and the winning sets $\cE$, the hypergraph $(X, \cE)$ is referred to as the \emph{hypergraph of the game}.
%Here Maker and Breaker alternately claim previously unclaimed vertices from $X$ and the goal of Maker is to claim all vertices of some edge in $\cE$.
%We refer to such a game as a \emph{hypergraph game}.
The following is a classical result in the theory of positional games.

\begin{theorem}[Erd\H os-Selfridge criterion \cite{ES73}]
\label{thm:erd_self}
Let $(X, \cE)$ be a hypergraph. Then, if Breaker has the first move in the game,
\begin{equation} \label{eq:erd_self}
\sum_{A \in \cE} 2^{-|A|} < 1
\end{equation}
is a sufficient condition for Breaker's win in the game $(X, \cE)$.
\end{theorem}

To see why this condition is sufficient, consider the following strategy for Breaker: choose $x\in X$ such that $\sum_{A \in \cE; x \in A} 2^{-|A|}$ is maximal, and denote with $\cE'$ the set of hyperedges which does not contain $x$. Then Maker's move will result in a vertex $y\in X$ such that
$\sum_{A \in \cE'; y \in A} 2^{-|A|}\le \sum_{A \in \cE; x \in A} 2^{-|A|}$. Observe that all edges $A\in \cE$ with $x\in A$ essentially disappear from the game, while the size of all edges $A\in \cE'$ with $y\in A$ just shrink by one. The choice of $x$ thus implies that the condition of the theorem remains valid and the theorem thus follows by induction.

%Even though we study the Maker-Breaker games in this paper, we will make use of the following result from the area of the so-called \emph{Avoider-Enforcer} games.
The following result guarantees that the first player cannot claim a cycle in the game played on the union of two disjoint forests.

\begin{theorem}[\cite{HKS07}] \label{thm:enforce_tree}
Let $F_1 = (V, E_1)$ and $F_2 = (V, E_2)$ be two edge disjoint forests on the same vertex set $V$. Then if two players alternately claim unclaimed edges from $E_1\cup E_2$, the second player can enforce that the edges of the first player span a forest.
\end{theorem}

Finally, the following result determines the threshold for the $K_3$-game.

\begin{theorem}[\cite{SS05}] \label{thm:K_3}
Consider the $K_3$-game (i.e.\ the \emph{triangle} game) played on the edge set of $\Gnp$. Then
$$\lim_{n \rightarrow \infty} \Pr[\Gnp \; \text{is Maker's win in the} \; K_3\text{-game}] =
\begin{cases}
1, &p \gg n^{-5/9}, \\
0, &p \ll n^{-5/9}.
\end{cases} $$
\end{theorem}

\subsection{Graph decompositions}

\begin{theorem}[Nash-Williams' arboricity theorem \cite{NW}] \label{thm:arb}
Any graph $G$ can be decomposed into $\lceil ar(G) \rceil$ edge-disjoint forests, where
$$ ar(G) = \max_{G' \subseteq G} \frac{e(G')}{v(G') - 1}. $$
\end{theorem}

The next lemma follows immediately from Hall's theorem. For convenience of the reader we add its short proof.

\begin{lemma} \label{lemma:orient}
The edges of any graph $G$ can be oriented such that the maximal outdegree is at most $\lceil m(G) \rceil$, where
$$ m(G) = \max_{G' \subseteq G} \frac{e(G')}{v(G')}. $$
\end{lemma}
\begin{proof}
Let $k:=\lceil m(G) \rceil$. We construct a bipartite graph $\hat G$ as follows. One vertex class consists of all edges of $G$ (class $P_e$) and the other of $k$ copies of each vertex of $G$ (class $P_v$). Furthermore, we add an edge between edge $e$ and a vertex $v$ if and only if $v$ is an endpoint of $e$ in $G$. It follows immediately from the definition of $m(G)$ and the construction of $\hat G$ that $\hat G$ satisfies Hall's condition with respect to the class $P_e$. Thus, $\hat G$ contains a matching $M$ that covers the set $P_e$. Orient an edge $e = \{v, u\}$ of $G$ towards $u$ if $\{e, v\}$ belongs to $M$ (for some copy of $v$ in $P_v$). Since each vertex appears only $k$ times in $P_v$, we deduce from the construction that the out-degree of each vertex is bounded by $k$. Since $M$ covers $P_e$, this process describes the orientation of every edge.
\end{proof}

\subsection{Hypergraph containers}
For the proof of the 1-statement of Theorem~\ref{thm:main}, we need the following consequence of the container theorems of Balogh, Morris, and Samotij~\cite{BMS12} and Saxton and Thomason~\cite{ST12}. The following theorem for all graphs $H$ is from~\cite{ST12}. A similar statement is obtained in~\cite{BMS12} for all $2$-balanced graphs~$H$.

\begin{definition}
For a given set $S$, let $\cT_{k, s}(S)$ be the family of $k$-tuples of subsets defined as follows,
$$ \cT_{k, s}(S) := \left\{ (S_1, \ldots, S_k) \,\Big| \,\,  S_i \subseteq S \; \text{for} \; 1 \leq i \leq k \; \text{and} \; \Big|\bigcup_{i=1}^k S_i \Big| \leq s\right\}. $$
\end{definition}

\begin{theorem}[\cite{ST12}, Theorem $1.3$] \label{thm:container}
For any graph $H$ there exist constants $n_0, s \in \mathbb{N}$ and $\delta < 1$ such that the following is true. For every $n\ge n_0$  there exists $t=t(n)$, pairwise distinct tuples $T_1,\ldots,T_{t} \in \cT_{s, sn^{2 - 1/m_2(H)}}(E(K_n))$ and sets
$C_1,\ldots,C_{t} \subseteq E(K_n)$, such that
\begin{enumerate}
\item[(a)] each $C_i$ contains at most $(1 - \delta)\binom{n}{2}$ edges,
\item[(b)] for every $H$-free graph $G$ on $n$ vertices there exists $1\le i\le t$ such that $T_i\subseteq E(G) \subseteq C_i$. (Here $T_i\subseteq E(G)$ means that all sets contained in $T_i$ are subsets of $E(G)$.)
\end{enumerate}
\end{theorem}

%Note that the main result in~\cite{ST12} is more general, as it provides a similar structure for independent sets in uniform hypergraphs.

\subsection{Random graphs}
\begin{theorem}[Markov's Inequality]
Let $X$ be a non-negative random variable. For all $t > 0$ we have $\Pr[X \geq t] \le \frac{\mathbb{E}[X]}{t}$.
\end{theorem}

\begin{theorem}[Chernoff's Inequality]
Let $X_1, \ldots, X_n$ be independent Bernoulli distributed random variables with $\Pr[X_i = 1] = p$ and $\Pr[X_i = 0] = 1 - p$. Then for $X = \sum_{i = 1}^nX_i$ we have $$\Pr[X \le (1 - \delta) \mathbb{E}[X]] \le e^{- \mathbb{E}[X] \delta^2 / 2}, \quad \text{for any} \quad 0 < \delta \le 1.$$
\end{theorem}

The following is a standard result from the random graph theory. We include its simple proof for convenience  of the reader.

\begin{lemma} \label{lemma:probabilistic2}
Let $\alpha, c,L$ be positive constants and assume $p \le cn^{-1/\alpha}$. Then w.h.p.\ every subgraph $G'$ of $\Gnp$ on at most $L$ vertices has density $m(G') \le \alpha$.
\end{lemma}
\begin{proof}
Observe that there exist only constantly many different graphs on $L$ vertices. Let $H$ be one such graph, and choose $\hat H \subseteq H$ such that $m(H) = e_{\hat H} / v_{\hat H}$. Then the expected number of $\hat H$-copies in $\Gnp$ is bounded by $n^{v_{\hat H}}p^{e_{\hat H}}$. Observe that for $p=cn^{-1/\alpha}$ we have $ n^{v_{\hat H}}p^{e_{\hat H}} =o(1)$ whenever $m(H) = e_{\hat H}/v_{\hat H} > \alpha$. It thus follows from Markov's inequality that for $p \le cn^{-1/\alpha}$ w.h.p.\ there is no $\hat H$-copy, and hence no $H$-copy in $\Gnp$. Therefore, it follows from the union bound that w.h.p.\ every subgraph $G'$ of $\Gnp$ of size $v_{G'}\le L$ satisfies $m(G') \leq \alpha$.
\end{proof}

Finally, in Section~\ref{sec:add} we use the following lemma that follows from a standard application of Chernoff's inequality.

\begin{lemma} \label{lemma:gnp_expand}
Let $p \gg \log n / n$ and $\varepsilon > 0$ be any constant. Then a graph $G := \Gnp$ satisfies w.h.p.\ the following property:
%\begin{enumerate}[(P1)]
for any subset $X \subseteq V(G)$ of size at most $1 / p$ we have
$$|N(X)| \geq (1 - \varepsilon) |X| n p,$$
%\item for any two disjoint subsets $X, Y \subset V(G)$ of size at least $\log n / p$, the number of edges between $X$ and $Y$ is at least $(1 - \varepsilon) |X||Y|p$.
%\end{enumerate}
\end{lemma}

\section{Proof of the $1$-statement of Theorem~\ref{thm:main}}\label{sec:main1}

Since we assume that Maker starts the game, the $1$-statement of Theorem~\ref{thm:main} follows directly from Theorem~\ref{thm:rr} and the strategy stealing argument. This argument can be easily augmented even for the case when Breaker starts, as the first move of Breaker typically cannot ruin the Ramsey property of the ground graph.

However, we would like to prove a strengthened version of part $(i)$ of Theorem~\ref{thm:main}, namely that a \emph{resilience}-type result also holds. In the proof we make use of the hypergraph containers, a new tool that seems to have potential for applications in positional games. A simplified version of this general approach was first utilized under a different name in~\cite{KS08}, where the following observation has been put to good use -- if there are two hypergraphs $\cH_1=(X, \cE_1)$ and $\cH_2=(X, \cE_2)$ such that every cover (set of vertices that intersects every hyperedge) of $\cH_1$ is also a cover of $\cH_2$, then a Breaker's win in the game played on $\cH_1$ implies a Breaker's win on $\cH_2$.

We note that the following theorem  can alternatively be proved using the approach of derandomized Maker's strategy from~\cite{BL00}, which is also well-suited for resilience-type results.

\begin{theorem}\label{thm:resil}
Let $H$ be any graph. Then there exist constants $C > 0$ and $\gamma > 0$ such that $G := \Gnp$ with probability $1 - e^{-\Theta(n^2 p)}$ satisfies the following: there exists a winning strategy for Maker in the $H$-game played on $E(G) \setminus R$, for any $R \subseteq E(G)$ with $|R| \le \gamma \cdot n^2 p$, provided that $p \geq Cn^{-1/m_2(H)}$.
\end{theorem}
\begin{proof}
Our proof is based on ideas of the proof from~\cite{NS13} of the $1$-statement of Theorem~\ref{thm:rr}. Note, however, that here we need to be much more careful: for the proof of  Theorem~\ref{thm:rr} one has to show that {\em every} coloring contains a monochromatic copy of $H$ in {\em some} color. Here we have to argue that we can find a {\em strategy} for Maker that ensures that he gets a monochromatic copy in {\em his} color. We achieve this by using the hypergraph game resp.\ Theorem~\ref{thm:erd_self}.

Let $\delta$ and $s$ be as given by Theorem~\ref{thm:container} when applied on the graph $H$. We prove the theorem for $\gamma = \delta/16$ and $C$ to be chosen later.

Let $G:= \Gnp$, and consider some subset $R \subseteq E(G)$ with $|R| \le \gamma \cdot n^2 p$. Observe that if Maker loses in the $H$-game on $E(G) \setminus R$, then by Theorem~\ref{thm:container} there exists $1 \leq i \leq t$ such that $T_i \subseteq E_M \subseteq C_i$, where $E_M$ is the set of Maker's edges.

Let us consider an auxiliary game played on the hypergraph $\cH = (E(G) \setminus R, \cE)$ with the vertex set being the edge set of $G \setminus R$ and the edge set
$$ \cE = \{ (E(K_n) \setminus C_i) \cap (G \setminus R) \; : \; T_i \subseteq G \setminus R \}. $$
In this game Breaker wins if he claims at least one edge from each set $(E(K_n) \setminus C_i) \cap (G \setminus R)$. Note that, by the previous observation, in case of Breaker's win the edge set of Breaker cannot be $H$-free.  We can thus conclude that Maker has a winning strategy in the $H$-game if he has a winning strategy (as Breaker) in the auxiliary game. In the light of Theorem~\ref{thm:erd_self} it remains to check that the hypergraph $(E(G) \setminus R, \cE)$ satisfies condition \eqref{eq:erd_self}.
First we show that all hyperedges typically have size at least $\delta n^2 p / 16$.  It follows from Theorem~\ref{thm:container} that $|E(K_n) \setminus C_i| \geq \delta \binom{n}{2} \geq \delta n^2 / 4$, for every $1 \leq i \leq t(n)$, and thus from Chernoff's inequality we have
\begin{equation} \label{eq:chernoff_C}
\Pr[|(E(K_n) \setminus C_i) \cap G| < \delta \cdot n^2 p /8] < e^{- \delta \cdot n^2 p / 32}.
\end{equation}
Let $\cB$ be the event that there exists a hyperedge which has less than $\delta n^2 p / 8$ vertices ``before'' the removal of $R$, i.e.
$$ \cB = \exists \; T_i \subseteq G \setminus R \; : \; |(E(K_n) \setminus C_i) \cap G| <  \delta n^2 p / 8. $$
Then
\begin{align*}
\Pr[\cB] &\leq \sum_{i=1}^{t(n)}\Pr[T_i \subseteq G \; \wedge \; |(E(K_n) \setminus C_i) \cap G| < \delta n^2 p / 8].
\end{align*}
As $T_i\subseteq C_i$, the two events are independent and we deduce
\begin{align*}
\Pr[\cB] &\leq \sum_{i=1}^{t(n)} \Pr[T_i \subseteq G] \cdot \Pr[|(E(K_n) \setminus C_i) \cap G| < \delta n^2 p / 8] \\
&\stackrel{\eqref{eq:chernoff_C}}{\leq} e^{- \delta n^2 p / 32} \cdot \sum_{i=1}^{t(n)} p^{|T_i^+|}, %= e^{- \delta n^2 p / 16} \cdot \mathbb{E}[|\cE|] \stackrel{\eqref{eq:num_of_hyperedges}}{\leq} e^{- \delta n^2 p / 16} \cdot 2^{\varepsilon n^2 p / 2},
\end{align*}
where $T_i^+$ is the union of all sets of the $s$-tuple $T_i$. Routine calculations (see~\cite{NS13} for details) imply that for any fixed $\varepsilon > 0$, by choosing $C$ sufficiently large (with respect to $s$ and $\varepsilon$), we have
\begin{equation}\label{eq:upper}
\sum_{i=1}^{t(n)} p^{|T_i^+|} \le 2^{\varepsilon n^2 p / 2}.
\end{equation}
Therefore, for a suitable chosen $\varepsilon$ (with respect to $\delta$), we have $\Pr[\cB] < e^{-\Theta(n^2 p)}$. It now easily follows that
$$\Pr[\exists \; A \in \cE \; : \; |A| < \delta \cdot n^2p / 16] = e^{-\Theta(n^2 p)},$$
regardless of the choice of $R$ (recall that we set $\gamma=\delta/16$). Finally, observe that for the expected number of edges we have
$$
\mathbb{E}[|\cE|] \leq \sum_{i = 1}^{t(n)}\Pr[T_i \subseteq G] = \sum_{i = 1}^{t(n)} p^{|T_i^+|} \stackrel{\eqref{eq:upper}}\le 2^{\varepsilon n^2 p / 2}.
$$
By Markov's inequality, we get
$$
\Pr[|\cE| \ge 2^{\varepsilon n^2 p}] \leq 2^{-\varepsilon n^2 p / 2}.
$$
Thus, with probability $1-o(1)$, $G$ is such that
$$ \sum_{A \in \cE} 2^{-|A|} \leq 2^{- \delta n^2 p / 32 + \varepsilon n^2 p} < 1$$
for $\varepsilon > 0$ small enough. Therefore, by Theorem~\ref{thm:erd_self}, Breaker has a winning strategy in the auxiliary game, hence by the previous discussion Maker has a winning strategy in the $H$-game played on $E(G) \setminus R$.
\end{proof}

\section{Criteria for Breaker's win in an $H$-game}\label{sec:main2}

In this section we collect some graph properties that suffice for characterizing the graph as a Breaker's win in an $H$-game. These will be used later in the proof of the $0$-statement of Theorem~\ref{thm:main}.

%First we prove that Breaker can make sure that the density of Maker's graph is bounded.
The following two criterions are fairly general and thus may be of independent interest.

\begin{proposition} \label{lemma:arb}
Let $G$ and $H$ be graphs such that
$$
\left\lceil \frac{ar(G)}{2} \right\rceil < ar(H),
$$
then Breaker  can win the $H$-game played on the edge set of $G$, even if Maker starts.
\end{proposition}

\begin{proof}
Let $k:=\left\lceil \frac{ar(G)}{2} \right\rceil$, and let $F_0, \ldots, F_{2k - 1}$ be the edge-disjoint decomposition of $G$ into forests which exists
by Theorem \ref{thm:arb}. Assume Breaker uses the strategy from Theorem \ref{thm:enforce_tree} for every pair of forests $F_{2i}$ and $F_{2i + 1}$, $0\le i < k$. Then Theorem \ref{thm:enforce_tree} implies
that Maker's edges can be partitioned into $k$ forests. Any subset $S$ of the vertex set can thus contain at most $k(|S|-1)$ Maker's edges. That is, the arboricity value for Maker's edges is at most~$k$ and, as $ar(H)>k$ by assumption, Maker's graph cannot contain $H$.
%It follows from Theorem \ref{thm:enforce_tree} that the second player can play such that the edges the first player claims from forests $F_{2i}$ and $F_{2i + 1}$ again span a forest, for each $i < k$. Let us denote every such forest with $M_i$, and consider a subset $S \subseteq V(G)$ of size $v$. The number of edges each forest contribute is at most $|M_i \cap S| - 1 \leq v - 1$, thus summing over all $M_i$ we get that $S$ induces at most $k \cdot (v - 1)$ Maker's edges.
\end{proof}

\begin{proposition} \label{lemma:orient_str}
Let $G$ and $H$ be graphs such that
$$
\left\lceil \frac{m(G)}{2} \right\rceil < m(H),
$$
then Breaker  can win the $H$-game played on the edge set of $G$, even if Maker starts.
\end{proposition}

\begin{proof}
Let us fix any orientation of the edges of $G$ such that each vertex has out-degree at most $\lceil m(G) \rceil$. Such an orientation exists by Lemma \ref{lemma:orient}. Now by a simple pairing strategy, it follows that Breaker can claim half of the outgoing edges of each vertex. In other words, the out-degree of each vertex, with respect to Maker's edges, is at most $\left\lceil \frac{\lceil m(G) \rceil}{2} \right\rceil = \left\lceil \frac{m(G)}{2} \right\rceil$. Therefore, by the condition of the proposition, the density of each subgraph of Maker's graph is less  than $m(H)$, and thus it cannot contain $H$ as a subgraph.
\end{proof}

With these two basic criteria at hand we can now prove the main theorem of this section.

\begin{theorem} \label{thm:deterministic}
Let $G$ and $H$ be graphs such that $m(G) \leq m_2(H)$ and $H$ is strictly $2$-balanced with at least $4$ vertices. Then Breaker has a winning strategy for the $H$-game on the edge set of $G$.
\end{theorem}
\begin{proof}
Let $m_2(H) = k + x$, for some $k \in \mathbb{N}$ and $0 \leq x < 1$. We first handle the case when $0 \leq x < 1/2$.

Since $H$ is strictly $2$-balanced we have
$$ m_2(H) = \frac{e_H - 1}{v_H - 2} > \frac{e_H - \delta(H) - 1}{v_H - 3}, $$
which easily implies $m_2(H) < \delta(H)$. For the sake of contradiction, let $G$ be the smallest graph such that Maker has a winning strategy. We first deduce that then $\delta(G) \geq 2(\delta(H) - 1) + 1$. Assuming otherwise, let $v$ be a vertex of degree at most $2(\delta(H) - 1)$. Then Breaker has the following winning strategy: whenever Maker claims an edge incident to $v$, Breaker does the same (if possible). If on the other hand Maker claims an edge from $G - \{v\}$, then Breaker follows his winning strategy for $G - \{v\}$ (which exists by choice of $G$). Then, clearly, Maker cannot build a copy of $H$ in $G - \{v\}$. Further, the degree of $v$ in the Maker's graph is at most $\delta(H) - 1$, thus it cannot be part of an $H$-copy either. Therefore, we have
$$ m(G) \geq \frac{\sum_{v \in G} \deg(v)}{2n} \geq \delta(H) - 1/2. $$
It now follows from $m_2(H) < \delta(H)$ that $\delta(H) \geq k + 1$ and thus $m(G) \geq k + 1/2$, which is a contradiction to $m(G) \leq m_2(H) < k + 1/2$.

From now on we can thus assume that $x \geq 1/2$. Next, we consider the case that $k \geq 3$. Observe that for every graph $H$ with at least $4$ vertices we have $\tfrac34v_H^2-v_H > \binom{v_H}{2} \geq e_H$, and thus
\begin{equation}
\frac{e_H}{v_H} + 3/2 > \frac{e_H - 1}{v_H - 2}. \label{eq:m_m2}
\end{equation}
%This can be seen from the following sequence of inequalities,
%\begin{align*}
% & \frac{e_H}{v_H} + 3/2 > \frac{e_H - 1}{v_H - 2} \\
% & \impliedby \frac{2e_H + 3v_H}{2v_H} > \frac{e_H - 1}{v_H - 2} \\
% & \impliedby (2e_H + 3v_H) (v_H - 2) > 2v_H (e_H - 1) \\
% & \impliedby 2e_H v_H - 4e_H + 3v_H^2 - 6v_H > 2v_H e_H - 2v_H \\
% & \impliedby \frac{3}{4} v_H^2 - v_H > \frac{v_H^2 - v_H}{2} \geq e_H.
%\end{align*}
Therefore $m(H) > m_2(H) - 3/2 \geq k - 1$, and so we have
$$ \lceil m(G) / 2 \rceil \leq \lceil (k+1) / 2 \rceil \stackrel{(k \geq 3)}{\leq} k - 1 < m(H).$$
Breaker's win now follows from Proposition~\ref{lemma:orient_str}.

If $H$ is not very dense, then a better estimate than the one in \eqref{eq:m_m2} can be made. In particular, $e_H < v_H^2 / 4$ implies that $\frac{e_H}{v_H} + 1/2 > \frac{e_H - 1}{v_H - 2}$. Since we also assumed that $x \geq 1/2$, this implies $m(H) > m_2(H) - 1/2 \geq k$. Similarly as before we have
$$\lceil m(G) / 2 \rceil \leq \lceil (k+1) / 2 \rceil \leq k < m(H), $$
and Breaker's win again follows from Proposition~\ref{lemma:orient_str}.

To summarize, so far we have shown that Breaker has a winning strategy for the $H$-game on graph $G$ if one of the following holds,
\begin{enumerate}[(a)]
\item $0 \leq x < 1/2$,
\item $k \geq 3$, or
\item $e_H < v_H^2 / 4$.
\end{enumerate}

Let us consider a graph $H$ which does not satisfy any of the above properties. Then $e_H \geq \lceil v_H^2 / 4 \rceil$ and thus 
$$m_2(H) = \frac{e_H - 1}{v_H - 2} \geq \frac{\lceil v_H^2 / 4 \rceil - 1}{v_H - 2} \geq 2$$
for $v_H \geq 5$, and since $H$ does not satisfy (a) and (b) we have $2.5 \leq m_2(H) < 3$. Furthermore, it is easy to check that $ar(G) \leq m(G) + 1/2$, and thus $ar(G) \leq m(G) + 1/2 \leq m_2(H) + 1/2 < 4$. On the other hand, from $m_2(H) \geq 2.5$ we have
$ e_H \geq \tfrac52 v_H -4$, and thus $e_H > 2v_H - 2 $ for $v_H \geq 5$, which implies $ar(H)> 2$.
It follows now from $\lceil ar(G) / 2 \rceil \leq 2 < ar(H)$ and Proposition~\ref{lemma:arb} that Breaker has a winning strategy  in this case.

Finally, checking all graphs on $4$ vertices we see that the only strictly $2$-balanced graphs are $K_4$ and $C_4$. The case $H = K_4$ is covered by Lemma 2.1 in \cite{MS}. For $H = C_4$ we have $ar(H) = 4/3$ and $ar(G) \leq m(G) + 1/2 \leq 2$, thus Proposition~\ref{lemma:arb} implies that Breaker has a winning strategy also in this case.
\end{proof}

\section{Proof of the $0$-statement of Theorem~\ref{thm:main}}\label{sec:main3}

We need to show that with high probability Breaker has a strategy such that, when played on the random graph $\Gnp$ with $p=c n^{-1/m_2(H)}$, for $0<c=c(H)<1$ small enough, Maker's edges do not span an $H$-copy. Observe that we may assume, without loss of generality, that $H$ is strictly $2$-balanced. If not, replace $H$ by a minimal subgraph $H'$ with the same $2$-density. Clearly, if  Breaker has a strategy for winning the $H'$-game on $\Gnp$, then the same strategy prevents  Maker from obtaining an $H$-copy.

Let us first give an intuition behind the Breaker's strategy. Observe that
the expected number of copies of $H$ on any given edge is bounded by
$$
v_H^2 \cdot n^{v_H-2}\cdot p^{e_H-1} = v_H^2 \cdot c^{e_H-1}.
$$
%as we assumed that $F$ its strictly $2$-balanced.
That is, for $0<c<1$ small enough we expect that the copies of $H$ are scattered 'loosely' and that we even have many edges that are not contained in any copy of $H$. Clearly, whether such edges are claimed by Maker or Breaker is irrelevant for the outcome of the game. Assume now we find a copy of $H$ that contains two edges which are not contained in any other copy of $H$. Then Breaker can easily ensure that this $H$-copy will never be claimed by Maker: fix two such edges arbitrarily and as soon as Maker claims the first of these edges, claim the other edge. Clearly, in this way this specific $H$-copy will never be a Maker's copy. We formalize these ideas as follows.

\begin{definition}We call an edge $\emph{free}$ if it does not belong to any copy of $H$, \emph{open} if it is contained in exactly one copy of $H$ and \emph{closed} otherwise. Furthermore, we call a copy of $H$ \emph{unproblematic} if it contains at least two open edges. Otherwise we call the copy \emph{problematic}.
\end{definition}

\noindent
\textbf{Preprocessing.}
Before starting the game, Breaker preprocesses the graph $G := \Gnp$ to obtain a subgraph $\hat G$ (with some special properties that we exhibit below) and a sequence of pairwise disjoint sets of edges $S_1,\ldots,S_k$ of cardinality two each:

\begin{list}{}{}
\item $i := 0$; $k=0$;
\item $G_i:= G;$
\item {\bf while} there exists an unproblematic copy $\hat H$ of $H$ in $G_i$
\item \hspace*{0.7cm} $k\leftarrow k+1$;
\item \hspace*{0.7cm} let $S_k\leftarrow \{$  two open edges (chosen arbitrarily) of $\hat H$  $\}$;
\item \hspace*{0.7cm} $i \leftarrow i + 1$;
\item \hspace*{0.7cm} $G_i \leftarrow G_{i-1} - \{$ all open edges of $\hat H\;\}$;

\item {\bf while} there exists a free edge $e \in G_i$
\item \hspace*{0.7cm} $i \leftarrow i + 1$;
\item \hspace*{0.7cm} $G_i \leftarrow G_{i-1} - e$;
\item $\hat G \leftarrow G_i$
\end{list}

\noindent
Note that within this algorithm open, free and closed are always defined with respect to the current graph $G_i$.

\noindent
\textbf{Strategy.} Assuming that Breaker has a winning strategy for the $H$-game when played on $\hat G$, the winning strategy for the whole graph $G$ is defined as follows: %by the remarks made above, the following strategy for Breaker is a winning strategy for the graph $\Gnp$:

\begin{list}{}{}
\item {\bf if} Maker claims an edge from $\hat G$
\item \hspace*{0.7cm} claim an edge from  $\hat G$ according to the winning strategy for $\hat G$;
\item {\bf else if} Maker claims an edge from a set $S_j$ for some $1\le j\le k$
\item \hspace*{0.7cm} claim the other edge from the set $S_j$;
\item {\bf else}
\item  \hspace*{0.7cm} take an arbitrary edge.%refrain from claiming an edge in this round (or take a random edge)
\end{list}

We first show that this strategy extends a winning strategy for $\hat G$ to a winning strategy for the whole graph.

\begin{claim} \label{claim:win}
Assuming that Breaker has a winning strategy for the $H$-game on $\hat G$, Breaker claims at least one edge from every copy of $H$ in $G$.
\end{claim}
\begin{proof}
First, consider an $H$-copy $\hat H$ which is contained in $\hat G$. Since Breaker is playing according to the winning strategy on $\hat G$, it follows that this copy has to contain at least one edge which belongs to Breaker. Secondly, consider an $H$-copy $\hat H$ which is contained in $G_i$ but not in $G_{i+1}$, for some $1 \leq i \le k$. It follows from the construction of $S_i$ that $S_i \subset \hat H$, and since Breaker claims at least one edge from $S_i$, he also claims at least one edge from $\hat H$.
\end{proof}

It remains to show that there exists a winning strategy for $\hat G$. In order to state the argument concisely, we introduce some notation.

\begin{definition}
An {\sl $H$-core} of $G$ is a maximal subgraph $G' \subseteq G$ (with respect to inclusion) that has the following two properties: every edge of $G'$ is contained in at least one copy of $H$ and every copy of $H$ in $G'$ is {problematic}.
\end{definition}

Recall that, by construction, $\hat G$ is an $H$-core. The following claim shows that it is the unique $H$-core.

\begin{claim}
There exists a unique $H$-core.
\end{claim}
\begin{proof}
Let us assume that there exist two different $H$-cores, say $G'$ and $G''$. Then $G' \not\subset G''$ and $G'' \not\subset G'$, so $G_s = G' \cup G''$ is a proper superset of $G'$ and $G''$. Therefore, to reach a contradiction to the maximality of $G'$ and $G''$ it suffices to show that $G_s$ is an $H$-core.

First, it is easy to see that every edge of $G_s$ is contained in at least one copy of $H$. Further, observe that every $H$-copy which is problematic in $G'$ or $G''$ remains problematic in $G_s$ as well. Thus, if an $H$-copy in $G_s$ is unproblematic then it cannot be contained in $G'$ nor in $G''$. Consider such an $H$-copy $\hat H$ and consider an arbitrary edge $e\in \hat H$. Then $e$ is contained in at least one of $G'$ and $G''$  and thus, by the definition of  $G'$ and $G''$, $e$ is also contained in a copy of $H$ different from $\hat H$. Therefore $e$ is closed in $G_s$, and thus $\hat H$ is problematic implying that $G_s$ is an $H$-core.
\end{proof}

We say that a subgraph $G'$ of the $H$-core of $G$ is {\sl $H$-closed} if every copy of $H$ from the $H$-core is either contained in $G'$
or edge-disjoint with $G'$. It is easy to see that the edges of the $H$-core can be partitioned into minimal $H$-closed subgraphs where minimal is with respect to subgraph inclusion. Furthermore, as all minimal $H$-closed subgraphs are edge disjoint, Breaker can consider each such subgraph independently.

The core of our argument is the following lemma which states that with high probability every minimal $H$-closed subgraph in the $H$-core of $\Gnp$ has constant size.

\begin{lemma} \label{lemma:probabilistic}
Let $H$ be a strictly $2$-balanced graph which is not a tree or a triangle. Then there exist constants $c > 0$ and $L > 0$ such that w.h.p.\ every minimal $H$-closed subgraph of the $H$-core of $\Gnp$ has size at most $L$, provided that $p \le cn^{-1/m_2(H)}$.
\end{lemma}

Before we prove Lemma~\ref{lemma:probabilistic}, we first show how it implies the $0$-statement of Theorem~\ref{thm:main}.

\begin{proof}[Proof of the $0$-statement of Theorem~\ref{thm:main}]
Let $G := \Gnp$, and let Breaker play as described. Recall that, by Claim~\ref{claim:win}, it suffices to show that there exists a winning strategy for the $H$-core $\hat G$ of $G$. Furthermore, by the definition of $H$-closed subgraphs, we only have to find a winning strategy for all minimal $H$-closed subgraphs of the $H$-core.

From Lemma~\ref{lemma:probabilistic} we know that w.h.p.\ the graph $G$ is such that all minimal $H$-closed subgraphs have size at most $L=L(H)$. From Lemma~\ref{lemma:probabilistic2} we know that  w.h.p.\ the graph $G$ is such that this implies that all minimal $H$-closed subgraphs have density at most $m_2(H)$. Theorem~\ref{thm:deterministic} thus implies that there exists a winning  strategy for  Breaker for all minimal $H$-closed subgraphs -- and thus also for the $H$-core $\hat G$, which together with Claim~\ref{claim:win} finishes the proof.
\end{proof}

It remains to prove Lemma~\ref{lemma:probabilistic}. We do this in the remainder of this section.

%\subsection{Proof of Lemma~\ref{lemma:probabilistic}} \label{subs:prob}
%
Actually, our proof of Lemma~\ref{lemma:probabilistic} follows the proof of Lemma~6 from~\cite{NS13}. The main difference is that in~\cite{NS13} a problematic copy of $H$ was defined as a copy of $H$ in which {\em all} edges are contained in two copies of $H$, while the definition in this paper allows the existence of one (but only one) edge that may be open. As we shall see, this difference in definition is responsible for the fact that the proof goes through for triangles in~\cite{NS13}, but does not here. Of course, this is no coincidence: for the Random Ramsey result that was considered in~\cite{NS13} the threshold for triangles is $p=n^{-1/m_2(K_3)}=n^{-1/2}$~\cite{LRV92}, while for the Maker-Breaker game considered in this paper the threshold for triangles is $n^{-5/9}$~\cite{SS05}.
In the following we repeat the main arguments from~\cite{NS13}, for the convenience of the reader. %so that the reader can check correctness and see the differences.
%For   the main Careful inspection yields that it works almost exactly the same, however, for the sake of completeness we give the whole proof.

We define a process that generates $H$-closed structures iteratively starting from a single copy of $H$. Assume that we have fixed an (arbitrary) total ordering $\omega$ of the edges of $\Gnp$, and let $G'$ be a minimal $H$-closed subgraph of the $H$-core of $\Gnp$. Then $G'$ can be generated by starting with an arbitrary $H$-copy in $G'$ and repeatedly attaching $H$-copies to the graph constructed so far, as described in the following procedure.

\begin{list}{}{}
\item Let $H_0$ be an $H$-copy in $G'$,
\item $k \leftarrow 0$; $\hat G \leftarrow H_0$;
\item {\bf while} $\hat G \not = G'$ {\bf do}
\item \hspace*{0.7cm}$k \leftarrow k+1$;
\item \hspace*{0.7cm}{\bf if} $\hat G$ contains a copy of $H$ that is unproblematic in $\hat G$ {\bf then}
\item \hspace*{1.4cm}let $\ell < k$ be the smallest index such that $H_{\ell}$ is
\item \hspace*{2.1cm} a copy of $H$ that is unproblematic in $\hat G$;
\item \hspace*{1.4cm}let $e$ be the $\omega$-minimum edge in $H_{\ell}$ which
\item \hspace*{2.1cm} is open in $\hat G$ and closed in $G'$;
\item \hspace*{1.4cm}let $H_{k}$ be an $H$-copy in $G'$ that contains $e$ but is
\item \hspace*{2.1cm} not contained in $\hat G$;
\item \hspace*{0.7cm}{\bf else}
\item \hspace*{1.4cm}let $H_{k}$ be an $H$-copy in $G'$ that is not contained
\item \hspace*{2.1cm} in $\hat G$ and intersects $\hat G$ in at least one edge;
\item \hspace*{0.7cm}$\hat G\leftarrow \hat G \cup H_{k}$;
\end{list}

In order to show that w.h.p.\ the highest value the parameter $k$ reaches is bounded by a constant, we first collect some properties of this process. Consider the $H$-copy $H_{i}$. We distinguish two cases: a) if $H_{i}$ intersects $\bigcup_{j< i}H_j$ in {\em exactly} one edge, we call this a \emph{regular} copy, and b) if $H_{i}$ intersects $\bigcup_{j< i}H_j$ in some subgraph $D$ with $v_D \ge 3$, we call this a \emph{degenerate} copy. Let us denote with $\reg(\ell)$ and $\deg(\ell)$ the number of $H$-copies $H_{i}$, $1\le i \leq \ell$, which are regular, resp. degenerate. Furthermore, for $0 \leq i \leq \ell$ we say that the copy $H_{i}$ is \emph{fully-open} at time $\ell$ if $H_{i}$ is a regular copy and no new vertex of $H_i$, i.e.,\ no vertex of $V(H_i)\setminus (\bigcup_{j < i} V(H_{j}))$, is touched by any of the copies $H_{i+1}, \ldots, H_\ell$. Let us denote with $f_o(\ell)$ the number of fully-open copies at time $\ell$. The following lemma implies that every fully-open copy at time $\ell$ contains exactly $e_H - 1$ open edges. % upon adding the copy $F_\ell$ to $\hat G$.

\begin{lemma}[Lemma $8$ in \cite{NS13}] \label{lemma:Hcopies}
Let $H$ be strictly $2$-balanced, let $G$ be an arbitrary graph and let $h_e$ be an edge of $G$. Construct a graph $G_H$ by attaching $H$ to an edge $h_e$. Then $G_H$ has the property that if $\hat H$ is an $H$-copy in $G_H$ that contains at least one vertex from $H$ that is not incident with $h_e$, then $\hat H = H$.
\end{lemma}

For $\ell \geq 1$, let
$$\Delta(\ell) := |\{ i < \ell \; : \; H_i\; \text{fully-open at time}\; \ell - 1, \; \text{ but not at time }\; \ell\}|.$$
Clearly, $\Delta(\ell) \leq 1$ if $H_\ell$ is a regular copy, and $\Delta(\ell) \leq v_H - 1$ if $H_\ell$ is a degenerate copy. The following claim is from~\cite{NS13} (Claim 10); the only difference is that we here have $e_H-3$ while in~\cite{NS13} we had $e_H-2$. (This difference comes from the fact the we now allow one open edge.)

\begin{claim}\label{claim:consec_reg}
For any sequence $H_i, \ldots, H_{i + e_H - 3}$ of consecutive regular copies such that $\Delta(i) = 1$ we have $\Delta(i+1) = \ldots = \Delta(i + e_H - 3) = 0$.\qed
\end{claim}
%\begin{proof}
%If $\Delta(i) = 1$, then $H_i$ intersects some fully-open copy $H_{i'}$ in exactly one edge. Since at that point the copy $H_{i'}$ had at least $e_H - 1$ open edges (and $e_H$if and only if $i = 1$), by the definition of the process the intersection with $H_i$ has to be on one of these edges. But then $H_{i'}$ has $e_H - 2$ remaining open edges and since it was chosen by the process at step $i$, it will be chosen in every consecutive step as long as it has at least two open edges. It follows from Lemma~\ref{lemma:Hcopies} that every regular copy closes at most one open edge, thus each of the copies $H_{i + 1}, \ldots, H_{i + e_H - 3}$ intersects $H_{i'}$ in an open edge which implies $\Delta(i + 1) = \ldots = \Delta(i + e_H - 3) = 0$.
%\end{proof}

Similarly, the next claim is proven exactly as Claim 11 in~\cite{NS13}, with $e_H-1$ (there) replaced by $e_H-2$ (here).%Next, we estimate the number of fully-open copies at time $\ell$ as a function of the number of regular and degenerate copies at that time.

\begin{claim}\label{claim:counting}
For every $\ell \geq 1$, assuming the process does not stop before adding the $\ell$-th copy, we have
$$ f_o(\ell) \geq \reg(\ell)\left(1 - \frac{1}{e_H - 2}\right) - \deg(\ell) \cdot v_H. $$
\qed
\end{claim}

Observe that this bound on $f_o(\ell)$ is only meaningful if $e_H \geq 4$. This
is the reason why the proof does not go through for the case of triangles.

If $f_o(\ell) > 0$ for some $\ell \geq 1$, then $H_\ell$ cannot be the last copy in the process, as there exists at least one $H$-copy with at least $e_H - 1 \geq 2$ open edges, which cannot be by the definition of the $H$-core. Furthermore, from Claim~\ref{claim:counting} we have that after adding $L$ copies, out of which at most $\xi$ are degenerate, there are still at least
\begin{equation} \label{eq:steps}
(L-\xi)(1 - 1/(e_H - 2)) - \xi \cdot v_F
\end{equation}
fully-open copies at time $L$.

In a first moment calculation we have to multiply the number of choices for $H_{\ell}$ with the probability that the chosen $H$-copy is in $\Gnp$. For a regular copy where $H_\ell$ is attached to an open edge, the open edge to which it is attached is given deterministically by the design of our algorithm, provided that $f_o(\ell) > 0$. We just have to choose the edge (and orientation) in the new copy that we attach to it. Thus, this term is bounded by
\begin{equation}\label{eq:rr:steps:e1}
2e_H \cdot n^{v_H-2} \cdot p^{e_H-1} \le 2e_H \cdot c < \tfrac12,
\end{equation}
for $0<c < 1/(4e_H)$. % (Recall that in the $0$-statement we assume $p\le cn^{-1/m_2(H)}$.) % where $v_F^2$ simply accounts the number of choices for the intersecting edge of $F_\ell$. %= c n^{-(v_F-2)/(e_F-1)}$.)
%Since we can choose $c$ to be as small as we want, we may assume $v_F^2 \cdot c < 1$.
For a regular copy $H_\ell$ that is either attached to a closed edge or to an open edge and $f_o(\ell) = 0$, the edge to which we attach the regular copy is not given deterministically so we have to choose two vertices to which we attach $H_\ell$, which we can do in at most $(\ell\cdot v_H)^2$ ways.

To bound the term for degenerate copies one first easily checks (see~\cite{NS13}) that %, observe first that for every subgraph $J\subsetneq F$ with $v_J\ge 3$ we have
%$$
% \tfrac{e_H-1}{v_H-2}= m_2(H) >  \tfrac{e_J-1}{v_J-2}\qquad\text{and thus}\quad  \tfrac{e_H-e_J}{v_H-v_J}=\tfrac{(e_H-1)-(e_J-1)}{(v_H-2)-(v_J-2)} > m_2(H).
%$$
there exists an $\alpha > 0$ such that
$$
(v_H-v_J) - \tfrac{e_H-e_J}{m_2(H)} < -\alpha,\qquad\text{for all $J\subsetneq H$ with $v_J\ge 3$. }
$$
Thus, we can bound the case that the copy $H_\ell$ is a degenerate copy by
\begin{equation}\label{eq:rr:steps:e2}
\sum_{J \subsetneq H, v_J \geq 3} (\ell\cdot v_H)^{v_J}\cdot n^{v_H-v_J} \cdot p^{e_H-e_J} < (\ell\cdot v_H \cdot 2^{e_H})^{v_H}\cdot  n^{-\alpha},
\end{equation}
with room to spare.

 % A regular step closes one edge (the edge $e$ to which it is attached) and adds $e_F-1$ new edges that are all open,  cf. Lemma~\ref{thm:rr:Fcopies}. A degenerate step can close many open edges simultaneously. In order to derive a bound on this number observe that  Lemma~\ref{thm:rr:Fcopies} implies that an edge $e\in F_i$ that was added during a regular step can only be closed by a degenerate step if some vertex of the $v_F-2$ new vertices of  $F_i$ was (ever) involved in a degenerate step. As a degenerate step can contain (with room to spare) vertices from at most $v_F$ different $F$-copies that were added during regular steps, we know that after $L$ steps out of which at most $\xi$ were degenerate steps, we
%still have at least
%$$
%(L-\xi) - \xi \cdot v_F
%$$
%$F$-copies that were added in a regular step so that its new vertices were never touched by a degenerate step. Note that Lemma~\ref{thm:rr:Fcopies} thus implies that in theses $F$-copies all edges except those to which another $F$-copy was added in a regular step  are still open. We thus have at least
%\begin{equation}\label{eq:rr:steps:e3}
%[(L-\xi) - \xi \cdot v_F ] \cdot (e_F-1) - L
%\end{equation}
%open edges. Clearly, this term is non-zero if $L=L(\xi)$ is sufficiently large.
%Choose $\xi$ such that $\xi\cdot \alpha > v_F +1$, and choose $L$ such that the term in (\refeq{eq:steps}) is positive. (Observe that $L$ is a constant that only depends on the forbidden graph $F$.) Finally, choose $\ell_0 = \frac{v_F+1}{|\log c|}\log n + \xi$, where $c$ is the constant from the definition of the edge probability~$p$.

With these preparations at hand we can now finish the proof exactly as in~\cite{NS13} by a  union bound argument, choosing $\xi$ such that $\xi\cdot \alpha > v_H +1$ and $L$ such that the term in (\refeq{eq:steps}) is positive. Informally, in \cite{NS13} it is shown that there are w.h.p.\ at most $\xi$ degenerate steps within the first $\Theta(\log n)$ steps. Furthermore, if the process doesn't stop before the $L$-th step then the term in (\refeq{eq:steps}) stays positive until at least $(\xi+1)$ degenerate steps occur, and by the previous observation this doesn't happen before the $\Theta(\log n)$-th step. Finally, we show that w.h.p.\ the process cannot run for $\Theta(\log n)$ steps. We skip the details.
%Observe that $L$ is a constant that only depends on the forbidden graph $H$. Finally, choose $\ell_0 ={(v_H+1)}\log_2 n + \xi + L$. %, where $c$ is the constant from the definition of the edge probability~$p$.
%
%Consider first all sequences with the property that $H_{\ell'}$, for $\ell' \le \ell_0$, is the $\xi$th degenerate copy. Then the expected number of subgraphs in $\Gnp$ that can be built by such a sequence is at most
%$$
%\sum_{\ell' \le \ell_0} \tbinom{\ell'-1}{\xi-1} n^{v_H} \cdot [ (\ell_0 v_H 2^{e_H})^{v_H}\cdot  n^{-\alpha}]^{\xi} \cdot (L v_H)^{2L} \cdot 2^{-(\ell' - L - \xi)} \le n^{v_H} \cdot o(n) \cdot n^{-\alpha\cdot \xi} =o(1),
%$$
%by choice of $\xi$. Here we used that regular copies contribute a term of less than $1/2$ if they occur after step $L$ and a term of at most $(L v_H)^2$ if they occur before.
%
%So we know that within the first $\ell_0$ copies we have less than $\xi$ degenerate ones. Then the choice of $L$ implies that the sequence that generates $G'$ either has length less than $L$ (which is fine) or length at least $\ell_0$. It thus suffices to consider all sequences of length $\ell_0$. The expected number of subgraphs in $\Gnp$ that can be built by such a sequence is at most
%$$
%\sum_{k < \xi} \tbinom{\ell_0}{k} n^{v_H} \cdot [ (\ell_0 v_H 2^{e_H})^{v_H}\cdot  n^{-\alpha}]^{k}  \cdot (L v_H)^{2L} \cdot 2^{-(\ell_0-L-k)} \le n^{v_H} \cdot o(n) \cdot n^{-(v_H+1)} =o(1),
%$$
%by choice of $\ell_0$. This concludes the proof of Lemma~\ref{lemma:probabilistic} and thus also the proof of the $0$-statement.

\section{Proof of Theorem~\ref{thm:k3_extension}}\label{sec:add}

In the following proof we use $M$ to denote Maker's graph.

\begin{proof}[Proof of Theorem~\ref{thm:k3_extension}]
If $m_2(H) \geq 2$, then  $H_P$ satisfies the condition of Theorem~\ref{thm:main}, and the conclusion of the theorem trivially follows. Therefore, we can assume that $m_2(H) < 2$.

Assume $p \ll n^{-t}$. If $t = 5/9$, then by Theorem~\ref{thm:K_3} Breaker can prevent Maker from creating a copy of $K_3$, and if $t = 1/m_2(H)<\frac{5}{9}$, then by Theorem~\ref{thm:main} Breaker can prevent Maker from creating a copy of $H$. In any case, there exists a subgraph of $H_P$ which Maker cannot create, thus Breaker wins in the $H_P$-game.

So, let now $p \gg n^{-t}$ and let $G := \Gnp$ be such that it satisfies the property given in Lemma~\ref{lemma:gnp_expand} with $\varepsilon = 1/2$. As $t>\frac{1}{2}$, without loss of generality we can add a technical assumption that $p \ll n^{-1/2}$. We split the strategy of Maker into several phases.

\smallskip

\noindent
\textbf{Phase 1.} Since $m(K_5^-) = 5/9$ (where $K_5^-$ is a complete graph on $5$ vertices with one arbitrary edge removed), $G$ contains w.h.p.\ a copy of $K_5^-$. Denote with $\hat K$ one such copy. It is not hard to check that playing only on the edges of $\hat K$, Maker can create a copy of $K_3$ in at most $4$ moves~\cite{SS05}. Let $K = \{v_1, v_2, v_3\}$ be the vertices of the obtained $K_3$-copy.

\noindent
\textbf{Phase 2.} It follows from Lemma~\ref{lemma:gnp_expand} that w.h.p.\ every vertex has at least $np/2 \gg 1 / (np^2)$ incident edges in $G$. Thus, in the next $8 / np^2$ rounds Maker can claim edges such that the set $N_1 = N_M(v_1) \setminus K$ has size  $8 / np^2$.

\noindent
\textbf{Phase 3.} Again, from Lemma~\ref{lemma:gnp_expand} and $1/p \gg |N_1|$ we have that w.h.p.\ $|N_{G}(N_1)| \ge \tfrac12 |N_1| np \ge 4/p$ and thus, with room to spare,  $|N_{G}(N_1) \setminus K| \geq 3 / p$. Therefore, regardless of Breaker's moves so far, in the next $1/p$ rounds Maker can claim edges such that the set $N_2 = N_M(N_1) \setminus K$ is of size $1/p$.

\noindent
\textbf{Phase 4.} It again follows from Lemma~\ref{lemma:gnp_expand} that $|N_{G}(N_2)| \geq \frac12 |N_2| np \geq n/2$ w.h.p. Again, regardless of Breaker's moves, in the next $n/6$ rounds Maker can easily claim edges such that the set $N_3 = N_M(N_2) \setminus (N_1 \cup K)$ is of size $n/6$.

\noindent
\textbf{Phase 5.} Maker creates a copy of $H$ in the induced subgraph $G[N_3]$.
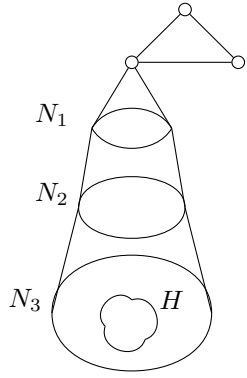
\begin{figure}[h!]
  \centering
  \tikzstyle{node}=[circle,fill=White,draw=Black,scale=0.5]
\tikzstyle{none}=[inner sep=0pt]
\tikzstyle{neutral}=[-]

\begin{tikzpicture}[scale=0.7]
	\begin{pgfonlayer}{nodelayer}
		\node [style=node] (0) at (0, 3) {};
		\node [style=node] (1) at (-1, 2) {};
		\node [style=node] (2) at (1, 2) {};
		\node [style=none] (3) at (-1.75, 0.75) {};
		\node [style=none] (4) at (-0.25, 0.75) {};
		\node [style=none] (5) at (-2, -0.75) {};
		\node [style=none] (6) at (0, -0.75) {};
		\node [style=none] (7) at (-2.5, -2.75) {};
		\node [style=none] (8) at (0.5, -2.75) {};
		\node [style=none] (9) at (-1, -2.5) {};
		\node [style=none] (10) at (-1.5, -3) {};
		\node [style=none] (11) at (-0.75, -3.25) {};
		\node [style=none] (12) at (-2.5, 1) {$N_1$};
		\node [style=none] (13) at (-2.5, -0.5) {$N_2$};
		\node [style=none] (14) at (-3, -2.5) {$N_3$};
		\node [style=none] (15) at (-0.25, -2.5) {$H$};
	\end{pgfonlayer}
	\begin{pgfonlayer}{edgelayer}
		\draw [style=neutral] (1) to (0);
		\draw [style=neutral] (0) to (2);
		\draw [style=neutral] (2) to (1);
		\draw [style=neutral, bend right=60] (3.center) to (4.center);
		\draw [style=neutral, bend left=60] (3.center) to (4.center);
		\draw [style=neutral] (1) to (3.center);
		\draw [style=neutral] (1) to (4.center);
		\draw [style=neutral, bend right=90] (5.center) to (6.center);
		\draw [style=neutral, bend left=90] (5.center) to (6.center);
		\draw [style=neutral] (3.center) to (5.center);
		\draw [style=neutral] (4.center) to (6.center);
		\draw [style=neutral, in=270, out=-90, looseness=1.25] (7.center) to (8.center);
		\draw [style=neutral, bend left=90, looseness=1.25] (7.center) to (8.center);
		\draw [style=neutral] (5.center) to (7.center);
		\draw [style=neutral] (6.center) to (8.center);
		\draw [style=neutral, bend left=90, looseness=1.50] (9.center) to (11.center);
		\draw [style=neutral, bend left=90, looseness=1.50] (11.center) to (10.center);
		\draw [style=neutral, bend left=90, looseness=1.50] (10.center) to (9.center);
	\end{pgfonlayer}
\end{tikzpicture}
  \caption{Evolution of Maker's graph in $H_p$-game.} \label{fig:H_p_proof}
\end{figure}

\smallskip

It remains to show that the last step (Phase 5) is indeed w.h.p.\ possible. First, observe that until this phase, only $o(n^2 p)$ rounds have been played. In other words, assuming that $n$ is sufficiently large, we know that less than $\tfrac{\gamma}{6^2} \cdot n^2 p$ rounds have been played to this point, where $\gamma$ is the constant given by Theorem~\ref{thm:resil}. On the other hand, it follows by a union bound that statement of Theorem~\ref{thm:resil} holds w.h.p.\ for every induced subgraph of $G$ on $n/6$ vertices,
\begin{gather*}
\Pr[\exists S \subseteq V(G) \; : \; |S| = n / 6, \; G[S] \; \text{does not satisfy Theorem~\ref{thm:resil}}\;] \leq \\
\leq \binom{n}{n/6} \cdot e^{- \Theta(n^2 p)} \leq e^{n - \Theta(n^2 p)} = o(1).
\end{gather*}
Therefore, we can assume that $G[N_3]$ satisfies the statement of Theorem~\ref{thm:resil}. Let $R \subset E(G)$ the set of Breaker's edges, and by previous observation we have $|R| \leq \tfrac{\gamma}{6^2} \cdot n^2 p$. Therefore, Maker can create a copy of $H$ in $G[N_3] \setminus R$, and by construction of set $N_3$ any such copy of $H$ closes a copy of $H_P$ in Maker's graph, see Figure~\ref{fig:H_p_proof}. This completes the proof of Theorem~\ref{thm:k3_extension}.
\end{proof}

%\begin{remark}
We close this section by mentioning that the phenomena of Theorem~\ref{thm:k3_extension} do hold for  $2$-connected graphs as well. For example, if we connect two vertices of the triangle by a path, then the threshold of the resulting graph will also depend on the length of this path. Let $C_3^+$ and $C_6^+$ be as defined in Figures~\ref{fig:C_3} and~\ref{fig:C_6}.

Adapting the proof of the $0$-statement of Theorem~\ref{thm:main} one can show that Breaker wins the $C_3^+$-game on $G_{n,p}$ whenever $p \le n^{-1/2-\varepsilon}$ for some $\varepsilon >0$. In addition, it follows from Theorem~\ref{thm:resil} that there exists a positive constant $C$ such that w.h.p.\ Maker has a winning strategy in the $C_3^+$-game, provided that $p \geq Cn^{-1/m_2(C_3^+)} = Cn^{-1/2}$.

For $C_6^+$ it follows from Theorem~\ref{thm:K_3} that Breaker can prevent Maker from obtaining a copy of $K_3$ (and thus of $C_6^+$ as well), whenever $p \ll n^{-5/9}$. On the other hand, adapting the ideas of the proof of Theorem~\ref{thm:k3_extension} one can show that  for $p \gg n^{-5/9}$ Maker has a winning strategy.

%\end{remark}

\begin{figure}[h!]
  \centering
  \begin{minipage}[t]{0.4\textwidth}
    \centering \tikzstyle{node}=[circle,fill=White,draw=Black,scale=0.5]
\tikzstyle{neutral}=[-]

\begin{tikzpicture}
	\begin{pgfonlayer}{nodelayer}
		\node [style=node] (0) at (0, 0.75) {};
		\node [style=node] (1) at (-1.25, -0.25) {};
		\node [style=node] (2) at (0, -1.25) {};
		\node [style=node] (3) at (1.25, -0.25) {};		
	\end{pgfonlayer}
	\begin{pgfonlayer}{edgelayer}
		\draw [style=neutral] (1) to (0);
		\draw [style=neutral] (0) to (3);
		\draw [style=neutral] (2) to (1);
		\draw [style=neutral] (2) to (3);
		\draw [style=neutral] (0) to (2);
	\end{pgfonlayer}
\end{tikzpicture}
    \caption{$C_3^+$ graph} \label{fig:C_3}
  \end{minipage}
  \begin{minipage}[t]{0.4\textwidth}
    \centering \tikzstyle{node}=[circle,fill=White,draw=Black,scale=0.5]
\tikzstyle{neutral}=[-]

\begin{tikzpicture}
	\begin{pgfonlayer}{nodelayer}
		\node [style=node] (0) at (-1, 0.75) {};
		\node [style=node] (1) at (-2.25, -0.25) {};
		\node [style=node] (2) at (-1, -1.25) {};
		\node [style=node] (3) at (1.5, 0.25) {};
		\node [style=node] (4) at (0.5, 1) {};
		\node [style=node] (5) at (1.5, -0.75) {};
		\node [style=node] (6) at (0.5, -1.5) {};
	\end{pgfonlayer}
	\begin{pgfonlayer}{edgelayer}
		\draw [style=neutral] (1) to (0);
		\draw [style=neutral] (0) to (4);
		\draw [style=neutral] (4) to (3);
		\draw [style=neutral] (3) to (5);
		\draw [style=neutral] (5) to (6);
		\draw [style=neutral] (6) to (2);
		\draw [style=neutral] (2) to (1);
		\draw [style=neutral] (0) to (2);
	\end{pgfonlayer}
\end{tikzpicture}
    \caption{$C_6^+$ graph} \label{fig:C_6}
  \end{minipage}
%  \begin{minipage}[t]{0.3\textwidth}
%  \input{figures/C_6_proof.tikz}
%   \caption{Evolution of Maker's graph in $C_6^+$-game.} \label{fig:C_6_proof}
%\end{minipage}
  \end{figure}

\bibliographystyle{acm}
\bibliography{H_game}

\end{document}